\documentclass[12pt]{article}

\usepackage{mathrsfs}
\usepackage{amssymb}
\usepackage{latexsym}
\usepackage{amsfonts}
\usepackage{amsmath}
\usepackage{eucal}
\usepackage{bm}
\usepackage{bbm}
\usepackage{graphicx}
\usepackage[english]{varioref}
\usepackage[nice]{nicefrac}
\usepackage[all]{xy}
\usepackage{amsthm}
\usepackage{amssymb,amsthm,upref,amscd}

\def\op{\operatorname}

\def\mmod{\kern-1pt\operatorname{-mod}}

\newtheorem{Thm}{Theorem}[section]
\newtheorem{Lem}[Thm]{Lemma}
\newtheorem{Cor}[Thm]{Corollary}
\newtheorem{Prop}[Thm]{Proposition}

\newtheorem{Rem}[Thm]{Remark}

\begin{document}
\renewcommand{\baselinestretch}{1.0}
\setlength{\parindent}{20pt} \baselineskip18pt

\centerline{\large \bf The Decomposition of Permutation Module}
\centerline{\large \bf for Infinite Chevalley Groups}

\bigskip

\centerline{Xiaoyu Chen, Junbin Dong*}

\begin{abstract}
Let ${\bf G}$ be a connected reductive group defined over $\mathbb{F}_q$, the finite field with $q$ elements. Let ${\bf B}$ be an Borel subgroup defined over $\mathbb{F}_q$. In this paper, we completely determine the composition factors of the induced module $\mathbb{M}(\op{tr})=\Bbbk{\bf G}\otimes_{\Bbbk{\bf B}}\op{tr}$  ($\op{tr}$ is the trivial ${\bf B}$-module) for any field $\Bbbk$.
\end{abstract}

\section{Introduction}
The representations of reductive algebraic groups is an interesting and fundamental topic. It has deep connections to other areas of mathematics, for example, algebraic geometry and number theory. The earlier attentions to this topic concentrated on the rational representations of algebraic groups and the representations of finite groups of Lie type.   The  cohomology theory of flag varieties and Deligne-Lusztig varieties control the rational representations of algebraic groups and ordinary representations of finite groups of Lie type, respectively.

One important class of irreducible modules of a reductive group (resp. Lie algebra) comes from certain induced modules from an one-dimensional character of a Borel subgroup (resp. Borel subalgebra).
For the rational representations of algebraic groups and, the representations of Lie algebras in the BGG category $\mathcal{O}$, it was known that all irreducible modules are the simple quotients of Weyl modules and Verma modules, respectively. Moreover, the decomposition of  Weyl modules and Verma modules motivates the famous Lusztig's conjecture (cf. \cite{Lu1} and \cite{Lu2}) and  Kazhdan-Lusztig conjecture (cf. \cite{KL}), respectively. For the representations of finite groups of Lie type in defining characteristic, such induced modules have been deeply investigated. For example, Carter and Lusztig classified simple modules via certain homomorphisms between such induced modules (cf.~\cite{CL}). Moreover, \cite{Jan} and \cite{Pil} indicated that the decomposition of such induced modules is closely related to the decomposition of Weyl modules.

Despite the fruitful results above, little was known about the abstract representations of algebraic groups. Assume that $\Bbbk$ is a field and let $\theta$ be a one-dimensional $\Bbbk{\bf B}$-module. It was observed in \cite{Xi} that the induced module $\mathbb{M}(\theta)=\Bbbk{\bf G}\otimes_{\Bbbk {\bf B}}\theta$ will give some new infinite dimensional abstract representations of ${\bf G}$. In particular, $\mathbb{M}(\op{tr})$ contains a submodule $\op{St}$ which is called infinite dimensional Steinberg module. The irreducibility of $\op{St}$ was proved in \cite{Xi} for the defining characteristic, and in \cite{Yang} for cross characteristic. Thus $\op{St}$ is irreducible for any field $\Bbbk$ which is surprising. For the nontrivial character $\theta$, it was proved in \cite{Chen1} that $\mathbb{M}(\theta)$ is irreducible if $\theta$ is strongly antidominant, and in \cite{Chen2} that a certain submodule of $\mathbb{M}(\theta)$ is irreducible when $\theta$ is antidominant.
Xi constructed in \cite{Xi} a filtration of $\mathbb{M}(\op{tr})=\Bbbk{\bf G}\otimes_{\Bbbk{\bf B}}\op{tr}$ whose subquotients are indexed by the subsets of simple reflections. The second author proved that some of these subquotients are irreducible when the groups are of type $A$ or rank $2$ in \cite{Dong} when $\op{char}\Bbbk\neq\op{char}\mathbb{F}_q$.
Later it was proved in \cite{CD} that all of these subquotients are irreducible and pairwise non-isomorphic if $\op{char}\Bbbk\neq\op{char}\mathbb{F}_q$.  This paper shows that the same result holds if $\op{char}\Bbbk=\op{char}\mathbb{F}_q$. Thus we completely determine the composition factors of $\mathbb{M}(\op{tr})$ for any field $\Bbbk$ (see Theorem \ref{main}). The constructions of these subquotients are uniform for all field, but the proof of irreducibility depends on the characteristic of $\Bbbk$. It would be interesting to find a characteristic free proof.

This paper is organized as follows: In Section 2 we recall some notations and basic facts about the structure of reductive groups. Section 3 recalls some basic properties of the  induced modules $\mathbb{M}(\op{tr})$. Section 4 gives the proof of the main theorem and in Section 5 we give another approach to prove our main theorem. Section 6 lists some open problems for further study.

\medskip
\noindent{\bf Acknowledgements.} Xiaoyu Chen is supported by National Natural Science Foundation of China (Grant No.~11501546). Junbin Dong is supported by National Natural Science Foundation of China (Grant No.~11671297).  The authors are grateful to Professor Nanhua Xi for his helpful suggestions and comments in
writing this paper.  Xiaoyu Chen thanks Professor Jianpan Wang and Naihong Hu for their advice and comments.
Junbin Dong thanks Professor Toshiaki Shoji and Qiang Fu for their helpful discussion and comments. Both authors would like to thank the referees for their careful reading and helpful suggestions.

\section{Reductive Groups with Frobenius Maps}
In this section, we recall the basic notations and facts about the structure of reductive groups. Let ${\bf G}$ be a connected reductive group defined over $\mathbb{F}_q$ with the standard Frobenius map $F$. Let ${\bf B}$ be an $F$-stable Borel subgroup, ${\bf T}$ be an $F$-stable maximal torus contained in ${\bf B}$, and ${\bf U}=R_u({\bf B})$ be the ($F$-stable) unipotent radical of ${\bf B}$. We denote $\Phi=\Phi({\bf G};{\bf T})$ the corresponding root system, and $\Phi^+$ (resp. $\Phi^-$) is the set of positive (resp. negative) roots determined by ${\bf B}$. Let $W=N_{\bf G}({\bf T})/{\bf T}$ be the corresponding Weyl group. For each $w\in W$, let $\dot{w}$ be a representative in $N_{\bf G}({\bf T})$.
One denotes $\Delta=\{\alpha_i\mid i\in I\}$ the set of simple roots and $S=\{s_i\mid i\in I\}$ the corresponding simple reflections in $W$.

For each $\alpha\in \Phi$, there is an unique unipotent subgroup
${\bf U}_\alpha$ of $G$ which is isomorphic to $\bar{\mathbb F}_q$ and is
stable under the conjugation by $\bf T$. For each $\alpha$, we fix an
isomorphism $\varepsilon_\alpha:\bar{\mathbb F}_q\to {\bf U}_\alpha$ so
that $t\varepsilon_\alpha(c)t^{-1}=\varepsilon_\alpha(\alpha(t)c)$.
For any $w\in W$, we set $$\Phi_w^-=\{\alpha \in \Phi^+ \mid w(\alpha)\in \Phi^- \}, \ \ \Phi_w^+=\{\alpha \in \Phi^+ \mid w(\alpha)\in \Phi^+ \}.$$
Now assume $\Phi_w^-=\{\beta_1,\beta_2,\dots,\beta_k\}$ and $\Phi_w^+=\{\gamma_1,\gamma_2,\dots,\gamma_l\}$
for a given $w\in W$ and, we denote
$${\bf U}_{w}= {\bf U}_{\beta_1}{\bf U}_{\beta_2}\dots {\bf U}_{\beta_k}\ \ \ \text{and}\ \ \  {\bf U}'_{w}= {\bf U}_{\gamma_1}{\bf U}_{\gamma_2}\dots {\bf U}_{\gamma_l}.$$
The following properties are well known (see \cite{Car}).

\smallskip

\noindent(a) For $w\in W$ and $\alpha\in \Phi$ we have $\dot{w}{\bf U}_\alpha \dot{w}^{-1}={\bf U}_{w(\alpha)}$;

\smallskip

\noindent(b) $ {\bf U}_w$ and ${\bf U}'_w$ are subgroups and $\dot{w}{\bf U}'_w\dot{w}^{-1} \subset {\bf U}$;

\smallskip

\noindent(c) The multiplication map ${\bf U}_w\times{\bf U}_w'\rightarrow{\bf U}$ is a bijection;

\smallskip

\noindent(d) Each $u \in {\bf U}_w$ is uniquely expressible in the form
$u=u_{\beta_1}u_{\beta_2}\dots u_{\beta_k}$ with $u_{\beta_i}\in  {\bf U}_{\beta_i}$;

\smallskip

\noindent(e) ({\it Commutator relations}) Given two positive roots $\alpha$ and $\beta$,  there exist a total ordering on $\Phi^+$ and integers $c^{mn}_{\alpha \beta}$ such that
$$[\varepsilon_\alpha(a),\varepsilon_\beta(b)]:=\varepsilon_\alpha(a)\varepsilon_\beta(b)\varepsilon_\alpha(a)^{-1}\varepsilon_\beta(b)^{-1}=
\underset{m,n>0}{\prod} \varepsilon_{m\alpha+n\beta}(c^{mn}_{\alpha \beta}a^mb^n)$$
for all $a,b\in \bar{\mathbb F}_q$, where the product is over all
integers $m,n>0$ such that $m\alpha+n\beta \in \Phi^{+}$, taken
according to the chosen ordering.

In the following sections, we will often use the properties of root subgroups. Except the properties above, we have the following  technical but useful lemma.

\begin{Lem}\label{Uw}
Let $s=s_{\alpha}$ be a simple reflection and $ws>w$. If ${\bf U}_w={({\bf U}_w)}^s$, then $ws=tw$ for some $t\in S$.
\end{Lem}
\begin{proof}
Let $\Phi_w^- = \{\alpha_1,\alpha_2,\dots,\alpha_m\}$.
Since ${\bf U}_{ws}={\bf U}_{\alpha}{({\bf U}_w)}^s={\bf U}_{\alpha}{\bf U}_w$, then we have
$$\Phi_{ws}^- = \{\alpha, \alpha_1,\alpha_2,\dots, \alpha_m\}.$$
Let $\Phi_w^+=\Phi^+\backslash \Phi_w^- =\{\beta_1=\alpha, \beta_2, \dots, \beta_l\}$. Denote by $\alpha'_i=w(\alpha_i)\in \Phi^-$ and $\beta'_i=w(\beta_i)\in \Phi^+$, then we have
$$\Phi^+=\{-\alpha'_1,\ -\alpha'_2,\ \dots,\ -\alpha'_m,\ \beta'_1,\ \beta'_2,\ \dots,\ \beta'_l\}.$$
Since ${\bf U}_w={({\bf U}_w)}^s$, there is a permutation $\sigma$ of $\{1,2,\dots, m\}$ such that $s(\alpha_i)=\alpha_{\sigma(i)}$. Therefore we have
$$wsw^{-1}(-\alpha'_i)=ws(-\alpha_i)=w(-\alpha_{\sigma(i)})=-\alpha'_{\sigma(i)}.$$
Similarly, there is a permutation $\tau$ of $\{2,3,\cdots,l\}$ such that
$wsw^{-1}(\beta'_j)=\beta'_{\tau(j)}$ for $j=2,3,\dots,l$.

The above discussion implies $\ell(wsw^{-1})\leq 1$.
But $wsw^{-1}\neq1$ and hence $wsw^{-1}=t\in S$ which completes the proof.
\end{proof}

For $J\subset I$, let $W_J$ be the standard parabolic subgroup of $W$ and assume that $w_J$ is the longest element in $W_J$.
For $w\in W$, set $\mathscr{R}(w)=\{s\in S\mid ws<w\}$ and denote
$$
\aligned
W^J &\ =\{x\in W\mid x~\op{has~minimal~length~in}~xW_J\};\\
Y^J &\ =\{w\in W^J\mid \mathscr{R}(ww_J)=J\}.
\endaligned
$$

\begin{Cor}\label{UwJ}
Let $s\in S$ and $w\in Y^J$. If $sw\in Y^J$ and $sw>w$, then ${\bf U}_{w_Jw^{-1}}\neq{({\bf U}_{w_Jw^{-1}})}^s$.
\end{Cor}
\begin{proof}
Suppose ${\bf U}_{w_Jw^{-1}}={({\bf U}_{w_Jw^{-1}})}^s$, then by Lemma \ref{Uw} there exists a simple reflection $r\in S$ such that $w_Jw^{-1}s=rw_Jw^{-1}$ which is a contradiction to $sw\in Y^J$.
\end{proof}

\section{The Permutation Module}
In this section, we recall some basic facts in \cite{Xi} and \cite{CD}. Assume that $\Bbbk$ is a field.  Let $\mathbb{M}(\op{tr})=\Bbbk{\bf G}\otimes_{\Bbbk{\bf B}}\op{tr}$, where $\op{tr}$ is the trivial $\Bbbk{\bf B}$-module, and call it the {\it permutation module}. Let ${\bf 1}_{\op{tr}}$ be a nonzero element in $\op{tr}$. For convenience, we abbreviate $x\otimes{\bf 1}_{\op{tr}}\in\mathbb{M}(\op{tr})$ to $x{\bf 1}_{\op{tr}}$. Since ${\bf T}$ acts trivially on ${\bf 1}_{\op{tr}}$, the notation
$w1_{tr}=\dot{w}1_{tr}$ is well defined for any $w\in W$.
Using the Bruhat decomposition of ${\bf G}$, it is easy to see
$$\mathbb{M}(\op{tr})=\sum_{w\in W}\Bbbk {\bf U}_{w^{-1}}w{\bf 1}_{\op{tr}}.$$
Moreover, the set $\{u{w}{\bf 1}_{\op{tr}}\mid w\in W, u\in {\bf U}_{w^{-1}}\}$ forms a basis of $\mathbb{M}(\op{tr})$.

\begin{Rem}
\normalfont Let $G={\bf G}^F$ and $B={\bf B}^F$. Naturally, we have a ``finite version" of $\mathbb{M}(\op{tr})$, namely, $\Bbbk G{\bf 1}_{\op{tr}}$, which is isomorphic to the induced module $\op{Ind}_{B}^{G}1_B$, where $1_B$ is the trivial $\Bbbk B$-module. For $\Bbbk=\mathbb{C}$, the decomposition of $\op{Ind}_{B}^{G}1_B$ is closely related to the representation of $\mathcal{H}=\op{End}_{G}(\op{Ind}_{B}^{G}1_B)$ which is known as the Hecke algebra. For $\Bbbk=\bar{\mathbb{F}}_q$, it is known that $\op{Ind}_{B}^{G}1_B$ decomposes into a direct sum of indecomposable modules, each with
simple socle, and there is a bijection between the direct summands and the subsets of $I$ (cf. \cite[Proposition 4.5]{YY}). However, we have
$$\op{End}_{\Bbbk{\bf G}}(\mathbb{M}(\op{tr}))\simeq\Bbbk$$
for any field $\Bbbk$, since it is clear that
$f({\bf 1}_{\op{tr}})\in \mathbb{M}(\op{tr})^{\bf U}=\Bbbk{\bf 1}_{\op{tr}}$. Therefore
the induced $\Bbbk{\bf G}$-module $\mathbb{M}(\op{tr})$ is indecomposable for any field $\Bbbk$.
\end{Rem}

For any $J\subset I$, let $W_J$  be the subgroup of $W$ generated by $s_i$ with $i\in J$. We set
$$\eta_J=\sum_{w\in W_J}(-1)^{\ell(w)}{w}{\bf 1}_{\op{tr}},$$
and let $\mathbb{M}(\op{tr})_J=\Bbbk{\bf G}\eta_J$. It was proved in \cite{Xi} that $\mathbb{M}(\op{tr})_J=\Bbbk {\bf U}W\eta_J$.
The following lemma is well known and very useful in our arguments later. The proof can be found in \cite[Proposition 2.3]{Xi} (see also \cite[Lemma 2.1]{CD}).

\begin{Lem}\label{xsty}
Let $u\in {\bf U}^*_{\alpha_i}= {\bf U}_{\alpha_i}\backslash\{1\}$ and $w\in W^J$. Then

\noindent $\op{(1)}$ There exist unique $x,y\in {\bf U}^*_{\alpha_i}$ and $t\in{\bf T}$ such that $\dot{s_i}u\dot{s_i}^{-1}=x\dot{s_i}ty$; Note that if we denote by $x=f_i(u)$, then $f_i$ is an isomorphism on ${\bf U}^*_{\alpha_i}$;

\noindent $\op{(2)}$ If $ww_J< s_iww_J$, then $\dot{s_i}u{w}\eta_J={s_i}{w}\eta_J$;

\noindent $\op{(3)}$ If $s_iw< w$, then $\dot{s_i}uw\eta_J=x{w}\eta_J$, where $x$ is defined in $\op{(1)}$.

\noindent $\op{(4)}$ If $s_iw> w$ and $s_iww_J< ww_J$, then $\dot{s_i}u{w}\eta_J=(x-1){w}\eta_J$, where $x$ is defined in $\op{(1)}$.
\end{Lem}

Since $\mathbb{M}(\op{tr})_J\supsetneq\mathbb{M}(\op{tr})_K$ if $J\subsetneq K$.
Following \cite[2.6]{Xi}, we define
$$E_J=\mathbb{M}(\op{tr})_J/\mathbb{M}(\op{tr})_J',$$
where $\mathbb{M}(\op{tr})_J'$ is the sum of all $\mathbb{M}(\op{tr})_K$ with $J\subsetneq K$. The following lemma was proved in \cite{Xi}.

\begin{Lem}[{\cite[Proposition 2.7]{Xi}}]\label{EJ}
If $J$ and $K$ are different subsets of $I$, then $E_J$ and $E_K$ are not isomorphic as
$\Bbbk{\bf G}$-modules.
\end{Lem}

We denote by $C_J$ the image of $\eta_J$ in $E_J$.  Combining \cite[Lemma 2.6]{Dong} and \cite[Lemma 2.7]{Dong} we see that
\begin{Prop}\label{YJCJ}
The set $\{u{w}C_J\mid w\in Y^J, u\in{\bf U}_{w_Jw^{-1}}\}$ forms a basis of $E_J$.
\end{Prop}

For any subset $J\subset I$, the $\Bbbk{\bf G}$-modules $E_J$ is a subquotient of $\mathbb{M}(\op{tr})$.
We can also realize $E_J$ as a $\Bbbk{\bf G}$-submodule of a parabolic induced module.
For $K\subset I$, let ${\bf P}_K$ be the standard parabolic subgroup of ${\bf G}$ generated by ${\bf B}$ and $s_i$ with $i\in K$, and $\mathbb{M}_K=\Bbbk{\bf G}\otimes_{\Bbbk{\bf P}_K}\op{tr}_K$, where $\op{tr}_K$ is the trivial ${\bf P}_K$-module. Then $\mathbb{M}_K$ is the quotient $\Bbbk{\bf G}$-module of $\mathbb{M}(\op{tr})$. Let ${\bf 1}_K$ be a nonzero element in $\op{tr}_K$. For convenience, we abbreviate $x\otimes{\bf 1}_K\in \mathbb{M}_K$ to $x{\bf 1}_K$.
For $J\subset I$, we denote $J'=I\backslash J$.  Let $E_J'$ be the $\Bbbk{\bf G}$-submodule of $\mathbb{M}_{J'}$ generated by $D_J:=\sum_{w\in W_J}(-1)^{\ell(w)}{w}{\bf 1}_{J'}$.   Combining Proposition \ref{YJCJ} and \cite[Proposition 3.2]{CD}, we get the following
\begin{Prop}\label{basis}
For any $J\subset I$, the set $\{u{w}D_J\mid w\in Y^J, u\in{\bf U}_{w_Jw^{-1}}\}$ forms a basis of $E_J'$. In particular, $E_J'\cong E_J$ as $\Bbbk{\bf G}$-modules.
\end{Prop}

\section{Composition factors of $\mathbb{M}(\op{tr})$ }

In this section we prove that $E_J$ is irreducible for any subset $J\subset I$. So we completely determines the composition factors of $\mathbb{M}(\op{tr})$ for any field $\Bbbk$.
The main theorem is the following
\begin{Thm}\label{main}
Let $\Bbbk$ be any field.  Then all the $\Bbbk{\bf G}$-modules $E_J$ $(J\subset I)$ are irreducible and pairwise non-isomorphic. In particular, $\mathbb{M}(\op{tr})$ has exactly $2^r$ composition factors, where $r$ is the rank of ${\bf G}$.
\end{Thm}

\begin{Rem}
\normalfont Theorem \ref{main} reflects a new phenomenon for infinite reductive groups. In other words, it does not hold when $\Bbbk{\bf G}$ is replaced by $\Bbbk G_{q^a}$.
When $\Bbbk=\mathbb{C}$, it is known that there is a bijection between the composition factors of $\Bbbk G_{q^a}{\bf 1}_{\op{tr}}$ and the composition factors of the regular module $\Bbbk W$ of $W$, which preserves multiplicities. But the number of composition factors of $\Bbbk W$ is not equal to $2^r$ in general.  When $\Bbbk=\bar{\mathbb{F}}_q$, let ${\bf G}=SL_3(\bar{\mathbb{F}}_q)$. Then Theorem \ref{main} says that $\mathbb{M}(\op{tr})$ has 4 composition factors. But it was shown in \cite{CL} (page 382) that $\Bbbk G_p{\bf 1}_{\op{tr}}$ has 6 composition factors, where $G_p= SL_3(\mathbb{F}_p)$.
\end{Rem}

Theorem \ref{main} was proved in \cite{CD} in the case $\op{char}\Bbbk\neq\op{char}\mathbb{F}_q$.
In this section we will prove Theorem \ref{main} in the case $\op{char}\Bbbk=\op{char}\mathbb{F}_q$. From here to the end of this section, we always assume that $\op{char}\Bbbk=\op{char}\mathbb{F}_q$.

For any finite subset $H$ of ${\bf G}$, let $\underline{H}:=\sum_{h\in H}h\in\Bbbk{\bf G}$ (this is a frequently used notation in the arguments below). It is clear that $\underline{H}\cdot\underline{H}=0$ if $H$ is a subgroup and $\op{char}\Bbbk$ divides $|H|$. For each $F$-stable subgroup ${\bf H}$ of ${\bf G}$, denote $H_{q^a}:={\bf H}^{F^a}$.

Although Theorem \ref{main} works for any field $\Bbbk$, the arguments in this paper are significantly different to that in the case $\op{char}\Bbbk\neq\op{char}\mathbb{F}_q$ in \cite{CD}. The following arguments, especially  Proposition \ref{Seperate} and Proposition \ref{Induct}, rely heavily on the condition $\op{char}\Bbbk=\op{char}\mathbb{F}_q$. While \cite[Lemma 2.4]{CD}, one of the key steps of arguments in \cite{CD}, relies heavily on the condition $\op{char}\Bbbk\neq\op{char}\mathbb{F}_q$. \cite[Lemma 2.4]{CD} says that for any ${\bf T}$-fixed nonzero element $\eta$ in a $\Bbbk{\bf G}$-module $M$, we have $\Bbbk{\bf G}\eta=\Bbbk{\bf G}\underline{U_q}\eta$ if $\op{char}\Bbbk\ne\op{char}\mathbb{F}_q$. However, this does not hold when $\op{char}\Bbbk=\op{char}\mathbb{F}_q$. For example, let $\Bbbk=\bar{\mathbb{F}}_q$, $M=\mathbb{M}(\op{tr})$, and $\eta={\bf 1}_{\op{tr}}\in M$. Then it is clear that $\Bbbk{\bf G} {\bf 1}_{\op{tr}}=\mathbb{M}(\op{tr})$, while $\Bbbk{\bf G}\underline{U_q} {\bf 1}_{\op{tr}}=0$ since $u {\bf 1}_{\op{tr}}= {\bf 1}_{\op{tr}}$ for any $u\in U_q$ and $\op{char}\Bbbk=\op{char}\mathbb{F}_q$.
Therefore, we cannot apply \cite[Lemma 2.4]{CD} to prove Theorem \ref{main} when $\op{char}\Bbbk=\op{char}\mathbb{F}_q$. So in this paper we use new ideas and techniques to deal with the defining characteristic case. Firstly we list two key technical results (Proposition \ref{Seperate} and Proposition \ref{Induct} below) used in the proof of Theorem \ref{main}.

\bigskip
For each nonempty subset $Y$ of $Y^J$, set $\Phi_Y=\bigcup_{w\in Y}\Phi_{w_Jw^{-1}}^-$. We fix a linear order on $\Phi_{Y^J}$ such that
$\Phi_{Y^J}=\{\beta_1,\cdots,\beta_m\}$ with $\op{ht}(\beta_1)\ge\cdots\ge\op{ht}(\beta_m)$, and assume that the linear order of each $\Phi_Y$ (In particular, each $\Phi_{w_Jw^{-1}}$) is inherited from $\Phi_{Y^J}$.

Let $a,b\in\mathbb{N}$ such that $a|b $. For each $w\in Y^J$, write $\Phi_{w_Jw^{-1}}^-=\{\gamma_1,\cdots,\gamma_t\}$ with respect to the above order (In particular $\op{ht}(\gamma_1)\geq\cdots\geq\op{ht}(\gamma_t)$). For such $a,b,w$, and $0\leq d\leq t$, set
$$\Theta(w,d,b,a):=\underline{U_{\gamma_1,q^b}}\cdots\underline{U_{\gamma_d,q^b}}\cdot
\underline{U_{\gamma_{d+1},q^a}}\cdots\underline{U_{\gamma_t,q^a}}.$$
With the above notations, we have the following proposition.

\begin{Prop}\label{Seperate}
Assume that $\op{char}\Bbbk=\op{char}\mathbb{F}_q$, and $M$ is a nonzero $\Bbbk{\bf G}$-module. Let $Y$ be a nonempty subset of $Y^J$ and write $\Phi_Y=\{\alpha_1,\cdots,\alpha_n\}$ with respect to the above order.
Let $d\in \mathbb{Z}_{\geq0}$ such that $\{\alpha_1,\dots,\alpha_d\}\subset\bigcap_{w\in Y}\Phi_{w_Jw^{-1}}^-$.
If $$0\neq\xi_d:=\sum_{w\in Y}a_w\Theta(w,d,b,a){w}C_J\in M$$ for $a,b\in \mathbb{N}$ such that $a\neq b$ and $a|b$ $($all $a_w\in\Bbbk$ here are nonzero$)$, then $\underline{U_{w_Jw^{-1},q^c}}{w}C_J\in M$ for some $w\in Y^J$ and $c\in \mathbb{N}$.
\end{Prop}

\begin{Prop}\label{Induct}
 Assume that $\op{char}\Bbbk=\op{char}\mathbb{F}_q$, and $M$ is a nonzero $\Bbbk{\bf G}$-module. If $\underline{U_{w_Jw^{-1}s,q^a}}{s}{w}C_J\in M$ for some $a\in \mathbb{N}$, where $sw\in Y^J$ and $sw>w$ $($this implies $w\in Y^J$$)$, then $\underline{U_{w_Jw^{-1},q^b}}{w}C_J\in M$ for some $b\in \mathbb{N}$.
\end{Prop}

\bigskip
\noindent Once Proposition \ref{Seperate} and Proposition \ref{Induct} are proved, we can prove Theorem \ref{main} in the case $\op{char}\Bbbk=\op{char}\mathbb{F}_q$ as follows.

\begin{proof} [{\it Proof of Theorem \ref{main}.}]
For a fixed $J\subset I$, assume that $M$ is a nonzero $\Bbbk{\bf G}$-submodule of $E_J$.
Let $E_{J,q^i}:=\Bbbk G_{q^i}C_J$. Choose a nonzero element $x\in M$.
Then $x\in E_{J,q^a}$ for some $a\in \mathbb{N}$ since $E_J=\bigcup_{i>0}E_{J,q^i}$.

It is clear that $(\Bbbk G_{q^a}x)^{U_{q^a}}\subset(E_{J,q^a})^{U_{q^a}}\subset\bigoplus_{w\in Y^J}\Bbbk\underline{U_{w_Jw^{-1},q^a}}{w}C_J$ by Lemma \ref{YJCJ}. Moreover, $(\Bbbk G_{q^a}x)^{U_{q^a}}\neq0$ by \cite[Proposition 26]{Se}. There exists a nonzero element
\begin{equation}\label{cw}
\xi=\sum_{w\in Y^J}c_w\underline{U_{w_Jw^{-1},q^a}}{w}C_J\in(\Bbbk G_{q^a}x)^{U_{q^a}}\subset M,\quad c_w\in\Bbbk.
\end{equation}

Choose an integer $b\neq a$ and $a|b$. Then $\xi=\sum_{w\in Y^J}c_w\Theta(w,0,b,a){w}C_J$. We apply Proposition \ref{Seperate} to $Y=\{w\in Y^J\mid c_w\ne 0\}$, $d=0$ and $\xi=\xi_d$. Then $\underline{U_{w_Jw^{-1},q^c}}{w}C_J\in M$ for some $w\in Y^J$ and $c\in \mathbb{N}$.
Applying Proposition \ref{Induct} repeatedly, we see that $\underline{U_{w_J,q^m}}C_J\in M$ for some $m \in \mathbb{N}$.

By \cite[Lemma 2]{St}, since $\op{char}\Bbbk=\op{char}\mathbb{F}_q$, we have
$$\sum_{w\in W_J}(-1)^{\ell(w)}{w}\underline{U_{w_J,q^m}}C_J
=\sum_{w\in W_J}q^{m\ell(w)}C_J=C_J\in M$$
which implies that $E_J$ is irreducible. The set $J$ in the above arguments can be any subset of $I$, so all $E_J$ $(J\subset I)$ are irreducible.
\end{proof}

\bigskip
\noindent Therefore, we devote to prove Proposition \ref{Seperate} and Proposition \ref{Induct} in the sequel. In order to prove these two propositions, we need the following technical lemma.

\begin{Lem}\label{key}
Fix $w\in Y^J$ and let $A=\{\alpha_1,\alpha_2,\dots, \alpha_m\}$ and $B=\{\beta_1,\beta_2,\dots, \beta_n\}$ be two disjoint subsets of $\Phi_{w_Jw^{-1}}^-$. Assume that $\sum_il_i\alpha_i\in A$ whenever $\sum_il_i\alpha_i\in\Phi^+$ for some $l_i\in\mathbb{Z}_{\geq0}$.
Let $a,b\in\mathbb{N}$ with $a|b$, and denote
$$\delta:=\underline{U_{\alpha_1,q^b}}\cdots\underline{U_{\alpha_m,q^b}}\cdot
\underline{U_{\beta_1,q^a}}\cdots\underline{U_{\beta_n,q^a}}wC_J.$$

\noindent Then we have

\noindent$\op{(i)}$ Assume that $k\beta_1+\sum_il_i\alpha_i\in A$ whenever $k\beta_1+\sum_il_i\alpha_i\in\Phi^+$ for some $k\in\mathbb{Z}_{>0}$ and $l_i\in\mathbb{Z}_{\geq0}$. Then
$$x\delta=\underline{U_{\alpha_1,q^b}}\cdots\underline{U_{\alpha_m,q^b}}\cdot
x \underline{U_{\beta_1,q^a}}\cdots\underline{U_{\beta_n,q^a}}wC_J$$
for any $x\in U_{\beta_1,q^b}$.

\noindent$\op{(ii)}$ Let $\gamma\in\Phi_{w_Jw^{-1}}^+$. Assume that $k\gamma+\sum_il_i\alpha_i+\sum_im_i\beta_i\in A$ whenever $k\gamma+\sum_il_i\alpha_i+\sum_im_i\beta_i\in\Phi_{w_Jw^{-1}}^-$ for some $k\in\mathbb{N}$ and $l_i,m_i\in\mathbb{Z}_{\geq0}$. Then $y\delta=\delta$ for any $y\in U_{\gamma,q^b}$.
\end{Lem}

\begin{proof}
$\op{(i)}$  By  commutator formula and the assumption, it is easy to show  that ${\bf V}_{A}=\displaystyle \prod_{1\leq i\leq m}{\bf U}_{\alpha_i} $ is a normal subgroup of ${\bf V}_A{\bf U}_{\beta_1}$. In particular, $V_{A,q^b}= \displaystyle  \prod_{1\leq i\leq m}U_{\alpha_i,q^b}$ is a normal subgroup of
$V_{A,q^b}U_{\beta_1,q^b}$. Thus, $x$ commutes with $\underline{U_{\alpha_1,q^b}}\cdots\underline{U_{\alpha_m,q^b}}$ which proves (i).

\medskip

\noindent $\op{(ii)}$ By assumption, for any $y\in U_{\gamma,q^b}$, $g_1\in V_{A,q^b}$, and $g_2\in V_{B,q^a}=\prod_{1\leq i\leq n}U_{\beta_i,q^a}$, we have
\begin{equation}\label{==}
yg_1g_2=\sigma(g_1)g_2z,
\end{equation}
where $\sigma(g_1)\in V_{A,q^b}$ and $z\in{\bf U}_{w_Jw^{-1}}'$. We claim that the map $g_1\mapsto\sigma(g_1)$ (for fixed $y$ and $g_2$) is injective. Indeed, assume that $\sigma(g_1)=\sigma(g_1')$ for some $g_1'\in V_{A,q^b}$. Since $yg_1'g_2=\sigma(g_1)g_2z'$ for some $z'\in{\bf U}_{w_Jw^{-1}}'$, we have
\begin{equation}\label{=}
g_2^{-1}g_1^{-1}g_1'g_2=z^{-1}z'.
\end{equation}
It follows from equation (\ref{=}) that $g_2^{-1}g_1^{-1}g_1'g_2\in {\bf U}_{w_Jw^{-1}}\cap{\bf U}_{w_Jw^{-1}}'=\{1\}$, and hence $g_1=g_1'$ which proves the claim. Since $zwD_J=wD_J$ for any $z\in{\bf U}_{w_Jw^{-1}}'$, we have $y\delta=\delta$ for any $y\in U_{\gamma,q^b}$ thanks to equation (\ref{==}) and the injectivity of $\sigma$.
\end{proof}

\noindent With this preparation in hand, we can give
\begin{proof}[Proof of Proposition \ref{Seperate}]
We will prove this lemma by the induction on $|Y|$.
If $|Y|=1$, then $\xi_d=c\Theta(w,d,b,a){w}C_J\in M$ for some $c\in\Bbbk^\times$ and $w\in Y^J$. We consider the $\Bbbk U_{q^b}$-module $N=\Bbbk U_{q^b} \Theta(w,d,b,a){w}C_J\subset M$. Clearly, $N^{U_{q^b}}\neq0$ by \cite[Proposition 26]{Se}. Note that $N^{U_{q^b}}\subseteq(\Bbbk U_{q^b}{w}C_J)^{U_{q^b}}= \Bbbk\underline{U_{w_Jw^{-1},q^b}}{w}C_J$, then $\underline{U_{w_Jw^{-1},q^b}}{w}C_J\in M$.

\smallskip
Assume that $|Y|>1$. Let $I_i$ be a set of left coset representatives of $U_{\alpha_i,q^a}$ in $U_{\alpha_i,q^b}$. Let $l$ be the minimal number such that $\alpha_{d+l}\not\in\Phi_{w_Jw^{-1}}^-$ for some $w\in Y$. Since $\Phi_{w_1}^-\neq\Phi_{w_2}^-$ if $w_1\neq w_2$, such $l$ always exists.

If $w\in Y$ and $\alpha_{d+l}\not\in\Phi_{w_Jw^{-1}}^-$, combining our assumption on the order in each $\Phi_{w_Jw^{-1}}^-$ and Lemma \ref{key} (i) yields
\begin{equation}\label{d+i+1}
\underline{I_{d+i+1}}\Theta(w,d+i,b,a){w}C_J
=\Theta(w,d+i+1,b,a)wC_J
\end{equation}
for all $0\le i<l-1$, and
Lemma \ref{key} (ii) yields
\begin{equation}\label{d+l}
\underline{I_{d+l}}\Theta(w,d+l-1,b,a){w}C_J=q^{b-a}\Theta(w,d+l-1,b,a){w}C_J=0
\end{equation}
since $\op{char}\Bbbk=\op{char}\mathbb{F}_q$ and $b\neq a$.
Thus, combining (\ref{d+i+1}) and (\ref{d+l}) yields
\begin{equation}\label{=0}
\underline{I_{d+l}}\cdots\underline{I_{d+1}}\Theta(w,d,b,a)wC_J=0.
\end{equation}

If $w\in Y$ and $\alpha_{d+l}\in\Phi_{w_Jw^{-1}}^-$, we have
\begin{equation}\label{!=0}
\underline{I_{d+l}}\cdots\underline{I_{d+1}}\Theta(w,d,b,a){w}C_J
=\Theta(w,d+l,b,a){w}C_J
\end{equation}
by Lemma \ref{key} (i). Denote $\xi_{d+l}:=\underline{I_{d+l}}\cdots\underline{I_{d+1}}\xi_d\in M$ and let $Y'$ be the set of $w\in Y^J$ such that the coefficient of $\Theta(w,d+l,b,a){w}C_J$ in $\xi_{d+l}$ is nonzero. Combining (\ref{=0}), (\ref{!=0}), and the minimality of $l$, we see that $\xi_{d+l}\neq 0$ (equivalently, $Y'$ is nonempty) and  $Y'\subsetneq Y$ (In particular $|Y'|<|Y|$). Notice that $\{\alpha_1,\cdots,\alpha_{d+l}\}\subset\bigcap_{w\in Y'}\Phi_{w_Jw^{-1}}^-$, The lemma follows from applying the induction hypothesis to $Y'$ and $\xi_{d+l}$.
\end{proof}

\bigskip

\begin{proof}[Proof of Proposition \ref{Induct}]
We may assume that the $a$ is big enough such that each $w\in W$ has a representative $\dot{w}$ in $G_{q^a}$. Fix a representative $\dot{s}$ of $s=s_{\alpha}$ in $G_{q^a}$.
Since ${\bf U}_{w_Jw^{-1}s}={\bf U}_{\alpha}{({\bf U}_{w_Jw^{-1}})}^s$, we have
$$\dot{s}\underline{U_{w_Jw^{-1}s,q^a}}{s}{w}C_J=\dot{s}\underline{U_{\alpha,q^a}}\dot{s}^{-1}
\underline{U_{w_Jw^{-1},q^a}}wC_J.$$
By Lemma \ref{xsty} (1), the above equation equals to
\begin{align*}
&(\underline{U^*_{\alpha,q^a}}\dot{s}+1)\underline{U_{w_Jw^{-1},q^a}}{w}C_J \\
= & \ \underline{U_{w_Jw^{-1}s,q^a}}{sw}C_J
+\underline{U_{w_Jw^{-1},q^a}}{w}C_J-\underline{({U_{w_Jw^{-1},q^a}})^s}{sw}C_J
\end{align*}
By the assumption $\underline{U_{w_Jw^{-1}s,q^a}}{sw}C_J\in M$, we get $$\underline{U_{w_Jw^{-1},q^a}}{w}C_J-\underline{({U_{w_Jw^{-1},q^a}})^s}{sw}C_J\in M.$$

\smallskip
Let $\Phi_{w_Jw^{-1}}^-\cap\Phi_{w_Jw^{-1}s}^-=\{\alpha_1,\alpha_2,\dots, \alpha_m\}$. By Corollary \ref{UwJ} we have ${\bf U}_{w_Jw^{-1}}\neq{({\bf U}_{w_Jw^{-1}})}^s$, which implies $\Phi_{w_Jw^{-1}}^-\cap\Phi_{w_Jw^{-1}s}^+\neq\emptyset$. Let $\Phi_{w_Jw^{-1}}^-\cap\Phi_{w_Jw^{-1}s}^+=\{\beta_1,\beta_2,\dots, \beta_n\}$. Hence  ${\bf U}_{w_Jw^{-1}}$ is the product of ${\bf U}_{\alpha_i}$ and ${\bf U}_{\beta_j}$ for $i=1,2,\dots,m$ and $j=1,2,\dots, n$. Write $\gamma_i=s(\beta_i)$, then $({\bf U}_{w_Jw^{-1}})^s$ is the product of ${\bf U}_{\alpha_i}$ and ${\bf U}_{\gamma_j}$ for $i=1,2,\dots,m$ and $j=1,2,\dots, n$.

\smallskip
Choose $\beta_{H}\in \{\beta_1,\beta_2,\dots, \beta_n\}$ such that
$$\text{ht}(\beta_{H})=\max \{\text{ht}(\beta_1),\text{ht}(\beta_2),\cdots,\text{ht}(\beta_n)\}.$$
Then the following  property hold: ($\clubsuit$) $\beta_{H}+\gamma_i\neq\gamma_j$ for any $i,j$.

\noindent Indeed, we have
$$w_Jw^{-1}s(\beta_H)=w_Jw^{-1}(\beta_H)-
\langle \beta_H,\alpha^\vee\rangle w_Jw^{-1}(\alpha) \in \Phi^+.$$
Since $w_Jw^{-1}(\beta_H) \in \Phi^-$ and $w_Jw^{-1}\alpha \in \Phi^+$, this forces $\langle \beta_H,\alpha^\vee\rangle <0$. If $\beta_{H}+\gamma_i=\gamma_j$, then
$$\beta_j=s(\beta_{H})+\beta_i=\beta_{H}+\beta_i- \langle \beta_{H},\alpha^\vee\rangle \alpha.$$
It follows that $\op{ht}(\beta_j)>\op{ht}(\beta_H)$ which contradicts to the choice of $\beta_H$. This proves Property ($\clubsuit$).

\medskip
We consider the following set
$$\displaystyle {\bf V}=\underset{\text{ht}(\alpha_i)\geq \text{ht}(\beta_{H})}{\prod}
{\bf U}_{\alpha_i}.$$
It is clear that ${\bf V}$ is a subgroup of ${\bf U}_{w_Jw^{-1}}$ and also a subgroup of $({\bf U}_{w_Jw^{-1}})^s$. Let
$${\bf V}_1=\prod_{{1\le i\le m}\atop{\op{ht}(\alpha_i)<\op{ht}(\beta_H)}}{\bf U}_{\alpha_i}\prod_{1\le i\le n}{\bf U}_{\beta_i},~~{\bf V}_2=\prod_{{1\le i\le m}\atop{\op{ht}(\alpha_i)<\op{ht}(\beta_H)}}{\bf U}_{\alpha_i}\prod_{1\le i\le n}{\bf U}_{\gamma_i}.$$
Then
${\bf U}_{w_Jw^{-1}}={\bf V} {\bf V}_1$ and ${({\bf U}_{w_Jw^{-1}})}^s= {\bf V} {\bf V}_2$. Let $b\in \mathbb{N}$ such that $b\neq a$ and $a|b$ and $I$ be a set of  the left coset representatives of $V_{q^a}$ in $V_{q^b}$, and write
$$\xi := \underline{I}\cdot(\underline{U_{w_Jw^{-1},q^a}}wC_J)\ \ \text{and} \ \
\eta :=\underline{I}\cdot\underline{({U_{w_Jw^{-1},q^a}})^s}swC_J.$$
We set
$${V_{1,q^a}}=\prod_{{1\le i\le m}\atop{\op{ht}(\alpha_i)<\op{ht}(\beta_H)}}{U}_{\alpha_i,q^a}\prod_{1\le i\le n}{U}_{\beta_i,q^a},~~{V}_{2,q^a}=\prod_{{1\le i\le m}\atop{\op{ht}(\alpha_i)<\op{ht}(\beta_H)}}{ U}_{\alpha_i,q^a}\prod_{1\le i\le n}{U}_{\gamma_i,q^a}.$$
It is clear that
$$\xi =\underline{V_{q^b}}\ \underline{V_{1,q^a}}wC_J \ \ \ \text{and} \ \ \
\eta = \underline{V_{q^b}} \ \underline{V_{2,q^a}}swC_J.$$
Since
$\underline{U_{w_Jw^{-1},q^a}}wC_J-\underline{({U_{w_Jw^{-1},q^a}})^s}swC_J\in M$,
we have $\xi-\eta\in M$.

\smallskip
Let $I_H$ be a set of the left coset representatives of $U_{\beta_{H},q^a}$ in $U_{\beta_{H},q^b}$. Using Property ($\clubsuit$) and Lemma \ref{key} (ii), we obtain $\underline{I_H}\eta= q^{b-a}\eta=0$ since $\op{char}\Bbbk=\op{char}\mathbb{F}_q$. Therefore by Lemma \ref{key} (i), $\underline{I_H}\xi\in M$ is nonzero. Let $N=\Bbbk U_{q^b}\underline{I_H}\xi\subset M$. Then $N^{U_{q^b}}\neq0$ by \cite[Proposition 26]{Se}. Since $N^{U_{q^b}}\subset(\Bbbk U_{q^b}{w}C_J)^{U_{q^b}}=\Bbbk\underline{U_{w_Jw^{-1},q^b}}{w}C_J$, we have $\underline{U_{w_Jw^{-1},q^b}}{w}C_J\in M$ which completes the proof.
\end{proof}

\section{Another Proof of Theorem \ref{main}}
In this section, we assume that $\op{char}\Bbbk=\op{char}\mathbb{F}_q$. Let $w_0$ be the longest element in $W$ and write $w_0=v_Jw_Jw_{J'}$ with $\ell(w_0)=\ell(v_J)+\ell(w_J)+\ell(w_{J'})$ (Recall that $J'=I\backslash J$). In this section, we combine Proposition \ref{Induct} and Proposition \ref{SeDJ} below to give an another proof of Theorem \ref{main}.
\begin{Prop}\label{SeDJ} Let $M$ be a nonzero $\Bbbk{\bf G}$-submodule of $E'_J$. Then $$\underline{U_{w_Jv_J^{-1},q^a}}{v_J}D_J \in M$$ for some $a\in \mathbb{N}$.
\end{Prop}

To prove this, we make some preparation. Following \cite[Proposition 3.16]{CL},  for any $a\in \mathbb{N}$ and $w\in W$ there is a $T_w\in\op{End}_{\Bbbk G_{q^a}}(\Bbbk G_{q^a}{\bf 1}_{\op{tr}})$ such that $T_w{\bf 1}_{\op{tr}}=\underline{U_{w,q^a}}{w^{-1}}{\bf 1}_{\op{tr}}$. For any $J\subset I$, denote
$$f^J_{q^a}=\sum\limits_{w\in w_0W_J}T_w{\bf 1}_{\op{tr}}=\sum\limits_{w\in w_0W_J}\underline{U_{w,q^a}}{w^{-1}}{\bf 1}_{\op{tr}}.$$
Combining \cite[Theorem 7.1]{CL}, \cite[Theorem 7.4]{CL}, and \cite[Corollary 7.5]{CL} yields
\begin{Lem}\label{simple}
The map $J\mapsto\Bbbk G_{q^a}f^J_{q^a}$ is a bijection between the subsets of $I$ and the irreducible summands of $\op{Soc}_{G_{q^a}}\Bbbk G_{q^a}{\bf 1}_{\op{tr}}$. Moreover, the stablizer of the space $\Bbbk f^J_{q^a}$ in $G_{q^a}$ is $P_{J,q^a}$, and $\Bbbk f^J_{q^a}$ is the unique 1-dimensional $U_{q^a}$-invariant space in $\Bbbk G_{q^a}f^J_{q^a}$.
\end{Lem}

Keep the notation ${\bf P}_K$, ${\bf 1}_K$, $\mathbb{M}_K$, $J'$ in the end of Section 3. For any $a\in \mathbb{N}$, let $\mathfrak{f}_{K,q^a}=\sum_{w\in W_K}\underline{U_{w^{-1},q^a}}{w}{\bf 1}_{\op{tr}}\in\mathbb{M}(\op{tr})$. Since $\mathfrak{f}_{K,q^a}$ is $P_{K,q^a}$-invariant and all $uw\mathfrak{f}_{K,q^a}$ $(w\in W^K, u\in U_{w^{-1},q^a})$ are linearly independent, the $\Bbbk G_{q^a}$-module $M_{K,q^a}=\Bbbk G_{q^a}{\bf 1}_K\subset\mathbb{M}_K$ is isomorphic to the $\Bbbk G_{q^a}$-submodule of $\Bbbk G_{q^a}{\bf 1}_{\op{tr}}$ generated by $\mathfrak{f}_{K,q^a}$ (via ${\bf 1}_K\mapsto\mathfrak{f}_{K,q^a}$).  The $\Bbbk G_{q^a}$-module $E_{J,q^a}'=\Bbbk G_{q^a}D_J$ is isomorphic to the submodule of $\Bbbk G_{q^a}{\bf 1}_{\op{tr}}$ generated by the element $\sum_{w\in W_J}(-1)^{\ell(w)}{w}\mathfrak{f}_{J',q^a}$ (via $D_J\mapsto\sum_{w\in W_J}(-1)^{\ell(w)}{w}\mathfrak{f}_{J',q^a}$). We denote $\varphi$ for this isomorphism in the sequel.

Since the conjugation by $w_0$ permutes the simple reflections, this induces a permutation $\sigma$ on $I$. Notice that $W_{\sigma J'}=w_0W_{J'}w_0$. By definition we have
$$
f_{q^a}^{\sigma J'}=\sum_{w\in W_{J'}w_0}\underline{U_{w,q^a}}{w^{-1}}{\bf 1}_{\op{tr}}=\sum_{w\in w_0W_{J'}}\underline{U_{w^{-1},q^a}}{w}{\bf 1}_{\op{tr}}.
$$
The above formula implies
\begin{equation}\label{fqj}
f_{q^a}^{\sigma J'}=\sum_{w\in W_{J'}}\underline{U_{w^{-1}w_Jv_J^{-1},q^a}}{v_J}{w_J}{w}{\bf 1}_{\op{tr}}=\underline{U_{w_Jv_J^{-1},q^a}}{v_J}{w_J}\mathfrak{f}_{J',q^a}.
\end{equation}
By the definition of $\varphi$, we have
\begin{equation}\label{fqj2}
\varphi\left(\sum_{w\in W_J}(-1)^{\ell(w)}\underline{U_{w_Jv_J^{-1},q^a}}v_Jw\mathfrak{f}_{J',q^a}\right)=\underline{U_{w_Jv_J^{-1},q^a}}v_JD_J.
\end{equation}
Assume that $w\lneqq w_J$.  Then there exists a $\gamma\in\Phi^+$ such that $w_Jv_J^{-1}(\gamma)\in \Phi^-$ and $w^{-1}v_J^{-1}(\gamma)\in \Phi^+$ and hence $$\underline{U_{\gamma,q^a}}v_Jw\mathfrak{f}_{J',q^a}=q^av_Jw\mathfrak{f}_{J',q^a}=0.$$
It follows that
\begin{equation}\label{fqj3}
\underline{U_{w_Jv_J^{-1},q^a}}v_Jw\mathfrak{f}_{J',q^a}=0
\end{equation}
if $w\lneqq w_J$. Combining (\ref{fqj}), (\ref{fqj2}), (\ref{fqj3}) yields $$\varphi(f_{q^a}^{\sigma J'})=(-1)^{\ell(w_J)}
\underline{U_{w_Jv_J^{-1},q^a}}{v_J}D_J\in E_{J,q^a}'.$$
\begin{Lem}\label{socEJ}
The $\Bbbk G_{q^a}$-socle of $E_{J,q^a}'$ is simple and generated by $\varphi(f_{q^a}^{\sigma J'})$.
\end{Lem}
\begin{proof}
By above discussion and Lemma \ref{simple}, $\op{Soc}_{G_{q^a}}E_{J,q^a}'\supset\Bbbk G_{q^a}\varphi(f_{q^a}^{\sigma J'})$. It remains to show that $\varphi(f_{q^a}^{\sigma K'})\not\in E_{J,q^a}'$ for $K\neq J$ by Lemma \ref{simple}. Suppose that $\varphi(f_{q^a}^{\sigma K'})\in E_{J,q^a}'$, then we have $D_K\in E_{J,q^a}'$ by the same arguments in the previous section, and the above discussion. It follows that $E_K'\subset E_J'$ and taking the ${\bf T}$-fixed points yields the inclusion $\phi:(E_K')^{\bf T}\rightarrow(E_J')^{\bf T}$. But $D_K\in (E_K')^{\bf T}$ is uniquely determined by the following two conditions: (i) $\dot{s_i}D_K=-D_K$ if and only if $i\in K$, and (ii) ${\bf U}_{\alpha_i}D_K=D_K$ if and only if $i\not\in K$. Therefore, $K\neq J$ implies any nonzero element in $(E_J')^{\bf T}$ does not satisfy the above conditions for $D_K$, and such $\phi$ does not exist. This contradiction completes the proof.
\end{proof}

\noindent With the above preparation, we can give
\begin{proof}[Proof of Proposition \ref{SeDJ}]
Let $0\neq x\in M$. Then $x\in M\cap E_{J,q^a}'$ for some $a\in \mathbb{N}$, and hence $$0\neq\Bbbk G_{q^a}x\supset\op{Soc}_{G_{q^a}}E_{J,q^a}'=\Bbbk G_{q^a}\varphi(f^{\sigma J'}_{q^a})$$ by Lemma \ref{socEJ}.
It follows that  $\varphi(f_{q^a}^{\sigma J'})=(-1)^{\ell(w_J)}\underline{U_{w_Jv_J^{-1},q^a}}{v_J}D_J\in M$ which completes the proof.
\end{proof}

Using Proposition \ref{SeDJ} and the same discussion  in Section 4, we can also prove that $E'_J$ is irreducible which implies the irreducibility of  $E_J$ by Proposition \ref{basis}.

\section{Further Developments}
In this section we propose some questions on infinite dimensional abstract representations of reductive groups with Frobenius maps. Any one-dimensional representation $\theta$ of ${\bf T}$ is regarded as a representation of ${\bf B}$ through the homomorphism ${\bf B}\rightarrow{\bf T}$. Let $\mathbb{M}(\theta)=\Bbbk{\bf G}\otimes_{\Bbbk{\bf B}}\theta$. If $\Bbbk=\bar{\mathbb{F}}_q$ and $\theta$ is a rational character of ${\bf T}$, the first author gave in \cite{Chen1} a necessary and sufficient condition for irreducibility of $\mathbb{M}(\theta)$, and found some $\mathbb{M}(\theta)$ with infinitely many irreducible subquotients. The following questions naturally arise.

\smallskip
\noindent $(1)$ Can one give a characteristic free proof of Theorem \ref{main}?

\smallskip
\noindent $(2)$ What is the necessary and sufficient condition for $\mathbb{M}(\theta)$ to have finitely many composition factors? If so, how does $\mathbb{M}(\theta)$ decompose?

\smallskip
\noindent $(3)$ Besides the irreducibility of $E_J$, Proposition \ref{basis} is more interesting in its own right. Now that $E_J$ can be realized as a submodule of a parabolic induced module, can one give a geometric construction of $E_J$ (probably using the geometry of partial flag varieties ${\bf G}/{\bf P}_K$, $K\subset I$)?

{\small
\noindent Xiaoyu Chen

\noindent E-mail: gauss\_1024@126.com

\noindent Department of Mathematics, Shanghai Normal University, 100 Guilin Road, Shanghai 200234,
P. R. China;

\bigskip
\noindent Junbin Dong

\noindent E-mail: dongjunbin1990@126.com

\noindent School of Mathematical Sciences, Tongji University, 1239 Siping Road, Shanghai 200092, P. R. China.
}
\end{document}